\documentclass[12pt,oneside]{amsart}

\usepackage{graphicx}
\usepackage{amsfonts}
\usepackage{epsf}
\usepackage{amssymb}
\usepackage{amsmath}
\usepackage{amscd}
\usepackage{tikz}
\usepackage{pdfpages}
\usepackage{fancyhdr}
\usepackage{setspace}
\usepackage{hyperref}
\usepackage[all]{xy}
\usetikzlibrary{matrix}
\usepackage{verbatim}
\usepackage{enumerate}

\theoremstyle{definition}
\newtheorem{theorem}{Theorem}[section]
\newtheorem{proposition}[theorem]{Proposition}
\newtheorem{lemma}[theorem]{Lemma}
\newtheorem{definition}[theorem]{Definition}
\newtheorem{question}[theorem]{Question}
\newtheorem{corollary}[theorem]{Corollary}

\newtheorem{remark}[theorem]{Remark}

\newcommand{\id}{\text{id}}

\def\co{\colon\thinspace}

\makeatletter
\newtheorem*{rep@theorem}{\rep@title}
\newcommand{\newreptheorem}[2]{%
\newenvironment{rep#1}[1]{%
 \def\rep@title{#2 \ref{##1}}%
 \begin{rep@theorem}}%
 {\end{rep@theorem}}}
\makeatother

\newreptheorem{theorem}{Theorem}
\newreptheorem{lemma}{Lemma}
\newreptheorem{question}{Question}
\newreptheorem{corollary}{Corollary}

\begin{document}

\rhead{\thepage}
\lhead{\author}
\thispagestyle{empty}


\raggedbottom
\pagenumbering{arabic}
\setcounter{section}{0}


\title{Surgery on tori in the 4--sphere}
\author{Kyle Larson}

\begin{abstract}
We investigate the operation of torus surgery on tori embedded in $S^4$. Key questions include which 4--manifolds can be obtained in this way, and the uniqueness of such descriptions. As an application we construct embeddings of 3--manifolds into 4--manifolds by viewing Dehn surgery as a cross section of a surgery on a surface. In particular, we give new embeddings of homology spheres into $S^4$.
\end{abstract}
\maketitle

\begin{section}{introduction}

Given an embedded torus $\mathcal{T}$ with trivial normal bundle in a 4--manifold $X$, torus surgery on $\mathcal{T}$ (also called a \emph{logarithmic transform}) is the process of removing a neighborhood $\nu \mathcal{T}$ and re-gluing $T^2 \times D^2$ by some diffeomorphism $\phi$ of the boundary to form $X_\mathcal{T} = X \setminus \nu \mathcal{T} \cup_\phi T^2 \times D^2$. Torus surgery is the operation underlying almost all examples of exotic 4--manifolds (see \cite{FS} for a nice overview). 
While torus surgery is a well-studied operation, most of the work has focused on tori embedded in elliptic surfaces (or at least in neighborhoods that admit a special elliptic fibration). Here we restrict to the case where the tori are embedded in $S^4$.

There are two natural 4-dimensional analogues to Dehn surgery on knots in $S^3$. The first is the Gluck twist operation \cite{gluck} on 2--knots, and the other is torus surgery (it is known that surgery on higher genus surfaces is a trivial operation since the gluing map will always extend over the tubular neighborhood of the surface). Now the possible 4--manifolds obtained by a Gluck twist in $S^4$ are quite limited; by an application of Freedman's theorem \cite{F} the result will always be homeomorphic to $S^4$. Furthermore, by a theorem of Iwase \cite{Iw1}, the result of a Gluck twist can also be obtained by a certain related torus surgery, and so for these reasons torus surgery seems to be the appropriate 4-dimensional generalization of Dehn surgery.

However, there is an immediate impediment to proving an analogue of the powerful Lickorish-Wallace theorem for Dehn surgery (which states that every closed, connected, oriented 3--manifold can be obtained by Dehn surgery on a link in $S^3$): torus surgery always preserves the Euler characteristic and signature of the 4--manifold. Therefore, the relevant question is:

\begin{question}\label{decomp}
Which 4--manifolds with Euler characteristic 2 and signature 0 can be obtained by surgery on a link of tori in $S^4$?
\end{question}

As a preliminary result in this direction we can show that a large class of groups can be obtained as fundamental groups of such 4--manifolds: 
we prove that any finitely presented group with non-negative deficiency appears as the fundamental group of a 4--manifold obtained by surgery on a link of tori in $S^4$. While we will see that it is also possible to obtain examples of groups with arbitrarily large negative deficiency, it is known that not all groups can be obtained in this way. Now at present a full answer to Question \ref{decomp} remains out of reach. However, a theorem by Baykur and Sunukjian \cite{BS} is relevant here. A consequence of their theorem is that any 4--manifold with Euler characteristic 2 and signature 0 can be obtained by a \emph{sequence} of torus surgeries starting in $S^4$ (in particular it may be necessary to have intermediate 4--manifolds). Question \ref{decomp} asks when is it possible to replace such a sequence with a single \emph{simultaneous} set of torus surgeries.

We will consider various spinning constructions to provide nice examples. We will produce an infinite family of distinct tori that admit non-trivial surgeries to $S^4$ (i.e. for each member of this family there exists a gluing map that does not extend to a diffeomorphism of $T^2\times D^2$, such that performing surgery with this gluing map results in a 4--manifold diffeomorphic to $S^4$), suggesting the possibility that manifolds obtained by a single torus surgery in $S^4$ never have a unique such description (i.e there always exist distinct tori that admit surgeries to the same manifold). We will also see that cyclic branched covers of spun knots can always be obtained by torus surgery in $S^4$, and that two spun knots are always related by torus surgery in their exteriors.

Much work has been done investigating which 3--manifolds embed in $S^4$ (for a few examples, see \cite{BB}, \cite{Crisp-Hillman}, \cite{Don}, \cite{Gil-Liv}). We construct embeddings of 3--manifolds into $S^4$ by considering cross sections of Gluck twists. The statement of our result is simplest if we start with a \emph{ribbon link} $L$ in $S^3$. If $M_L$ is the 3--manifold obtained by surgery on $L$ with all the surgery coefficients belonging to the set $\{1/n\}_{n \in \mathbb{Z}}$, then we have that $M_L$ smoothly embeds in $S^4$. The resulting 3--manifold is an integral homology sphere, and so we see that this theorem allows us to construct embeddings for a large family of integral homology spheres into $S^4$. We also examine surgery on the unknotted torus, and use this to show that the 3--manifolds obtained by $p/q$ Dehn surgery on a knot in $S^3$  always embed in either $S^1 \times S^3 \# S^2 \times S^2$ or $S^1 \times S^3 \# S^2 \widetilde{\times} S^2$. If we puncture the 3--manifold then we can eliminate the $S^1 \times S^3$ connected summand. 

\begin{subsection}{Organization}
In Section \ref{basics} we give definitions and consider the basic algebraic invariants related to torus surgery. A discussion of several spinning constructions and their connection to torus surgery takes place in Section \ref{spinexamples}. Our results regarding fundamental groups appear in Section \ref{round} in the context of interpreting torus surgeries as round cobordisms. Lastly, Section \ref{unknot} contains our results about surgery on the unknotted torus and embeddings of 3--manifolds.
\end{subsection}

\subsection{Acknowledgments}

The author would like to thank his advisor Robert Gompf for helpful comments and conversations.  The author was partially supported by NSF grant DMS-1148490.

\end{section}

\begin{section}{the basics}\label{basics}

We will assume that all manifolds and maps are smooth, and that homology is calculated with integer coefficients unless otherwise noted.

\begin{subsection}{Torus exteriors}

A \emph{surface knot} $K$ is an embedded submanifold in $S^4$ that is diffeomorphic to some closed surface. When $K$ is diffeomorphic to $S^2$ it is called a 2--knot. This paper is concerned with the case that $K$ is diffeomorphic to the torus $T^2$, and we will simply say that $K$ is a torus in $S^4$ (henceforth we switch to the notation $\mathcal{T}$ for a torus in $S^4$). Let $E_{\mathcal{T}} = \overline{S^4 \setminus \nu \mathcal{T}}$ denote the \emph{exterior} of $\mathcal{T}$. We can compute the homology of $E_\mathcal{T}$ by the long exact sequence of the pair $(S^4, E_\mathcal{T})$, using the isomorphism $H_n(S^4, X_\mathcal{T}) \cong H_n(\nu\mathcal{T}, \partial \nu \mathcal{T})$ from excision. The result is that $H_n(E_\mathcal{T})$ is isomorphic to $\mathbb{Z}$ for $n=0,1$, to $\mathbb{Z} \oplus \mathbb{Z}$ for $n=2$, and to the trivial group otherwise. The calculation shows that $H_1(E_\mathcal{T})$ is generated by the homology class of a meridian of $\mathcal{T}$ and generators of $H_2(E_\mathcal{T})$ are given by the \emph{rim tori} $S^1 \times \{pt\} \times \partial D^2$ and $\{pt\} \times S^1 \times \partial D^2$ in $\partial E_\mathcal{T} = \partial \nu \mathcal{T}$ under the identification $\nu \mathcal{T} = S^1 \times S^1 \times D^2$ (and these tori have algebraic intersection number 0 in $E_\mathcal{T}$).

The fundamental groups of torus exteriors in $S^4$ (and for surface knots in general) have been widely studied. The collection of such groups includes all 2--knot groups, and hence all classical knot groups (for an overview see \cite{CKS}). Among other things, it is known that this collection contains groups of arbitrarily large negative deficiency \cite{Levine}. (The \emph{deficiency} of a finite group presentation is the number of generators minus the number of relations. The deficiency of a group is the maximum deficiency of all presentations for the group.)

\end{subsection}


\begin{subsection}{Torus surgery}\label{definitions}

Let $\mathcal{T}$ be an embedded torus in $S^4$. We want to think of $\mathcal{T}$ as a particular embedding of $S^1 \times S^1$ into $S^4$, so that we have fixed curves $\alpha = S^1 \times \{pt\}$ and $\beta = \{pt\} \times S^1$ in $\mathcal{T} \subset S^4$ (whose homology classes provide a preferred basis for $H_1(\mathcal{T})$). Note that it is possible for there to be infinitely many distinct isotopy classes of embeddings $S^1 \times S^1 \hookrightarrow S^4$ with the \emph{same} submanifold as their image \cite{Hirose}. A framing for $\mathcal{T}$ is a particular identification of a tubular neighborhood $\nu\mathcal{T}$ with $T^2 \times D^2$. Given our fixed embedding, there is a canonical framing for $\mathcal{T}$, specified by requiring that the pushoffs $\alpha \times \{pt\}$ and $\beta \times \{pt\}$ in $T^2 \times \partial D^2$ are nullhomologous in the exterior $E_{\mathcal{T}}$ of $\mathcal{T}$. (Framings are identified with $H^1(T^2) \cong \mathbb{Z} \oplus \mathbb{Z}$, and since the first homology of the exterior is generated by a meridian we can twist the $D^2$ factor along $\alpha$ and $\beta$ so that the pushoffs are nullhomologous.) 

We now define the operation of interest in this paper. \emph{Torus surgery on $\mathcal{T}$} is the process of removing $\nu \mathcal{T}$ from $S^4$ and re-gluing $T^2 \times D^2$ by a diffeomorphism $\phi \co T^2 \times \partial D^2 \rightarrow T^2 \times \partial D^2$, using our canonical framing to identify $\partial \nu \mathcal{T}$ with $T^2 \times \partial D^2$. We will momentarily denote the resulting closed 4--manifold by $S^4_\mathcal{T}(\phi)$. Since $T^2 \times D^2$ admits a handle decomposition relative to its boundary with one 2-handle, two 3-handles, and a 4-handle, we can construct $S^4_\mathcal{T}(\phi)$ from $E_\mathcal{T}$ by adding one 2-handle, two 3-handles, and a 4-handle. There is a unique way to attach 3- and 4-handles for a closed 4-manifold (\cite{LP}, \cite{Mont}), and so $S^4_\mathcal{T}(\phi)$ is determined up to diffeomorphism by the attaching circle of the 2-handle (the framing must be the product framing). The attaching circle will be the image of the meridian $\{pt\} \times \partial D^2$ under $\phi$, and this is determined up to isotopy by its homology class 
$[\phi(\{pt\} \times \partial D^2)] = p[m] + a[\alpha] + b[\beta]$, where $m$ is the meridian of $\mathcal{T}$. Therefore, given our fixed embedding of $\mathcal{T}$ and the resulting canonical framing, $S^4_\mathcal{T}(\phi)$ is determined up to diffeomorphism by the integers $p$, $a$, and $b$.
Hence we will denote a torus surgery on $\mathcal{T}$ by $S^4_\mathcal{T}(p,a,b)$. It turns out that the integer $p$ is particularly important, and it is called the \emph{multiplicity} of the surgery. If we think of $T^2 \times \partial D^2$ as $\mathbb{R}^3 / \mathbb{Z}^3$, then we can represent our gluing map $\phi$ by a matrix in $GL(3, \mathbb{Z})$. Since the resulting diffeomorphism type only depends on the image of $\{pt\} \times \partial D^2$, we can choose the gluing map to be any integral matrix (with determinant $\pm 1$) of the form:
\[\phi = \left( \begin{array}{ccc}
* & * & a \\
* & * & b \\
* & * & p \end{array} \right)\]

Now there is another common notation to specify a particular torus surgery. The homology class $a[\alpha] + b[\beta] \in H_1(T^2)$ equals $q\gamma$ for some primitive element $\gamma \in H_1(T^2)$.  We call $q$ the \emph{auxiliary multiplicity} and $\gamma$ the \emph{direction} of the surgery. Then specifying the multiplicity, auxiliary multiplicity, and direction determines the resulting diffeomorphism type of the surgery. If $q=1$ we will say the surgery is \emph{integral}.

The trivial surgery is $S^4_\mathcal{T}(1,0,0)$, which returns $(S^4, \mathcal{T})$. Note that if we choose a different embedding of $\mathcal{T}$ (thought of as a submanifold), we will get the same set of possible torus surgeries but the surgery data could be different. We will also consider surgery on links of tori (a collection of multiple disjoint embeddings of tori into $S^4$), and the resulting manifold will be determined by the surgery data for each individual torus surgery.

\end{subsection}

\begin{subsection}{Algebraic topology}

Next we examine the basic algebraic topology of $S^4_\mathcal{T}(p,a,b)$. Since we have already computed the homology of $E_\mathcal{T}$, we can compute the homology of $S^4_\mathcal{T}(p,a,b)$ using the long exact sequence of the pair $(S^4_\mathcal{T}(p,a,b), E_\mathcal{T})$. For this calculation it is useful to change our identification $\nu \mathcal{T} = T^2 \times D^2$ by a self-diffeomorphism of $T^2 \times D^2$ that is the identity on the second factor but on the $T^2$ factor sends the direction $\gamma$ to $[\{pt\} \times S^1]$. Then we can choose our gluing map $\phi \co T^2 \times \partial D^2 \rightarrow T^2 \times \partial D^2$ to be:
\[\phi = \left( \begin{array}{ccc}
1 & 0 & 0 \\
0 & c & q \\
0 & d & p \end{array} \right)\] 
for some $c$ and $d$ satisfying $cp-dq = 1$. Following the calculation we see that for $p \neq 0$, $H_n(S^4_\mathcal{T}(p,a,b))$ is isomorphic to $\mathbb{Z}_p$ for $n=1,2$, and vanishes for $n=3$. Furthermore, $H_1(S^4_\mathcal{T}(p,a,b))$ is generated by the original meridian $m$ in $E_\mathcal{T}$ and $H_2(S^4_\mathcal{T}(p,a,b))$ is generated by the glued-in torus $T^2 \times \{0\}$. In particular, we observe that multiplicity 1 surgery produces a homology 4--sphere.

Similar computations show multiplicity 0 surgery results in a 4--manifold with the homology of $S^1 \times S^3 \# S^2 \times S^2$.

There is a simple relationship between the fundamental group of $S^4_\mathcal{T}(p,a,b)$ and the fundamental group of the torus exterior $E_\mathcal{T}$.
We start with a presentation of $\pi_1(E_\mathcal{T})$ and add a single relation corresponding to the attaching circle of the 2-handle.

\end{subsection}


\begin{subsection}{Spin structures}


Recall that a 4--manifold $X$ is spin if and only if its second Stiefel-Whitney class $w_2(X) \in H^2(X; \mathbb{Z}_2)$ vanishes. If $X$ is spin, the set of distinct spin structures can be identified with $H^1(X; \mathbb{Z}_2)$. For \emph{odd} multiplicity $p$, we can calculate from the integral homology of $S^4_\mathcal{T}(p,a,b)$ that $ H^2(S^4_\mathcal{T}(p,a,b); \mathbb{Z}_2) \cong  H^1(S^4_\mathcal{T}(p,a,b); \mathbb{Z}_2) \cong 0$. Hence $S^4_\mathcal{T}(p,a,b)$ has a unique spin structure for odd $p$, regardless of the particular choice for $a$ or $b$. 

For even $p$ the situation is more subtle. Here we follow Iwase \cite{Iw1}. Suppose we take a curve on $\mathcal{T}$ and push off using our canonical framing to obtain a curve $c$ in $\partial E_\mathcal{T}$ such that $[c] = 0$ in $H_1(E_\mathcal{T}; \mathbb{Z}_2)$. Let $c'$ be a pushoff of $c$ in $\partial E_\mathcal{T}$ using the product framing of the boundary. Now let $D$ and $D'$ be 2--chains in $E_\mathcal{T}$ such that $[\partial D] = [c]$ and $[\partial D'] = [c']$ (mod 2), and $D$ and $D'$ intersect transversely. Then $q([c]) = D \cdot D'$ (mod 2) is a well-defined function. In fact, $q$ is the Rokhlin quadratic form \cite{R} for $\mathcal{T}$ and so it satisfies $q([c_1] + [c_2]) = q([c_1]) + q([c_2]) + [c_1] \cdot [c_2]$ (mod 2). Furthermore, if the kernel of the inclusion map $H_1(\partial E_\mathcal{T}; \mathbb{Z}_2) \rightarrow H_1(E_\mathcal{T}; \mathbb{Z}_2)$ is $\{0, e_1, e_2, e_3\}$, then Iwase shows that $q(e_i) = 1$ for exactly one $e_i$ (in other words the Arf invariant is 0 for tori in $S^4$; see also \cite{R}). This motivates the following definition.

\begin{definition}
A particular embedding $S^1 \times S^1 \hookrightarrow S^4$ will be called a \emph{spin embedding} if $q([\{pt\} \times S^1]) = q([S^1 \times \{pt\}]) = 0$. We will say the resulting torus $\mathcal{T}$ in $S^4$ is \emph{spin embedded}.
\end{definition}

Note that we can always change our embedding of a torus so that it is spin embedded. Now we can determine when the result of an even multiplicity surgery is spin. Iwase \cite{Iw1} worked this out for a special class of tori obtained by spinning torus knots in $S^3$, and in fact his proof works in this more general context.

\begin{proposition}[\cite{Iw1}]\label{spinprop}
$S^4_\mathcal{T}(p,a,b)$ is spin if $p$ is odd. If $p$ is even, assume that $\mathcal{T}$ is spin embedded. Then $S^4_\mathcal{T}(p,a,b)$ is spin if and only if $ab = 0$ (mod 2).
\end{proposition}

\begin{proof}
We saw above that $S^4_\mathcal{T}(p,a,b)$ is spin if $p$ is odd, so assume $p$ is even. The Mayer-Vietoris sequence with $\mathbb{Z}_2$ coefficients (we use these coefficients for the rest of the argument) gives us:
\begin{equation}
H_2(T^2 \times D^2) \oplus H_2(E_\mathcal{T}) \xrightarrow{\Psi} H_2(S^4_\mathcal{T}(p,a,b)) \xrightarrow{\partial} H_1(\partial(T^2 \times D^2)) \xrightarrow{\Phi} H_1(T^2 \times D^2) \oplus H_1(E_\mathcal{T}) \nonumber
\end{equation}
Now the kernel of $\Phi = \{0, [m] \}$ for a meridian $m$ of $\mathcal{T}$. Hence we have an induced split exact sequence and isomorphism $H_2(S^4_\mathcal{T}(p,a,b)) =  im \Psi \oplus \langle [D_m + D_\sigma] \rangle$ for $D_m$ the class of the meridinal disk in $T^2 \times D^2$ and $D_\sigma$ a 2--chain in $E_\mathcal{T}$ bounded by the surgery curve $\sigma  = \phi (\{pt\} \times \partial D^2)$ (whose homology class $p[\mu] + a[\alpha] + b[\beta]$ is 0 since $p$ is even). Now the $\mathbb{Z}_2$ intersection form 
is trivial on $ im \Psi$ and on $\langle [D_m + D_\sigma] \rangle$ we have $[D_m + D_\sigma]^2 = q([\sigma]) = a^2q(\alpha) + b^2q(\beta) + ab$ (mod 2). If $\mathcal{T}$ is spin embedded then this equals $ab$ (mod 2) and so by the Wu formula we get that $w_2(S^4_\mathcal{T}(p,a,b)) = 0$ (and hence $S^4_\mathcal{T}(p,a,b)$ is spin) if and only if $a b = 0$ (mod 2).

\end{proof}

\end{subsection}

\begin{subsection}{Special neighborhoods}

The fishtail neighborhood $F$ and cusp neighborhood $C$ are compact 4-manifolds that admit elliptic fibrations over the disk with a single fishtail or cusp singular fiber, respectively. We can describe handle decompositions for these manifolds as follows (see Figure \ref{cuspandfishtail}). We start with a handle decomposition for $T^2 \times D^2$, and to form $F$ we add another 2-handle attached along a pushoff of an $S^1$ factor of $T^2 \times \{0\}$, where the framing of the  2-handle will be obtained from the product framing of the boundary by adding a single left-handed twist. To form $C$ we add one more 2-handle along a pushoff of the other $S^1$ factor of $T^2 \times \{0\}$, where again the framing will be given by taking the product framing and adding a single left-handed twist. The attaching circles for these extra 2--handles are called \emph{vanishing cycles}.


\begin{figure}
\includegraphics[scale=0.25]{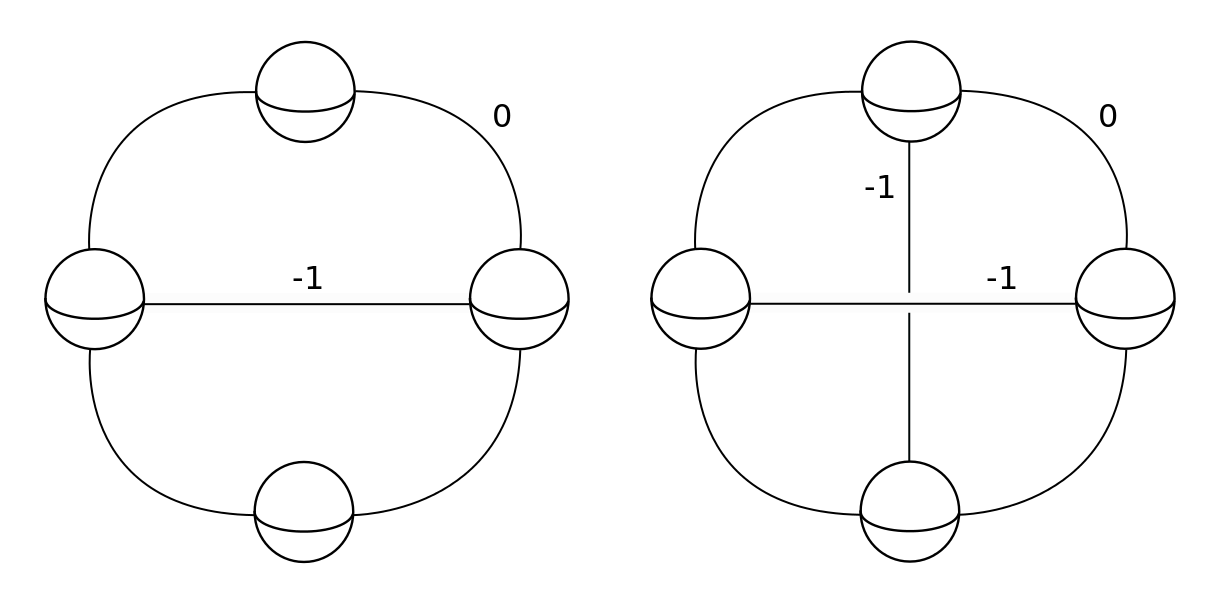}
\centering
 \caption{Here are Kirby diagrams for the fishtail neighborhood $F$ (left) and the cusp neighborhood $C$ (right).}
\label{cuspandfishtail}
\end{figure}

Performing torus surgery on regular fibers of fishtail and cusp neighborhoods has been an important operation in the theory of 4--manifolds. In this context we have a fixed framing for the torus fiber coming from the fibration map, and so we have well-defined notions of the multiplicity, auxiliary multiplicity, and direction of the surgery as before. Here we state two theorems that demonstrate nice properties satisfied by torus surgeries in these neighborhoods. For proofs we refer the reader to \cite{GS} for the first theorem and to \cite{Gom1} for the second.

\begin{theorem}
The result of performing torus surgery on a regular fiber of a cusp neighborhood $C$ depends only on the multiplicity $p$ of the surgery, up to diffeomorphism relative to the boundary.
\end{theorem}

This theorem says that the result of torus surgery in $C$ is independent of the auxiliary multiplicity or direction of the surgery. We observe that $C$ does not admit an embedding into $S^4$. For example, by Proposition \ref{spinprop} we see that whether the result of performing an even multiplicity surgery on a torus in $S^4$ is spin depends on the auxiliary multiplicity and the direction. If such a torus was a regular fiber of a cusp neighborhood there could be no such dependence. However, there do exist embeddings of fishtail neighborhoods into $S^4$, and we will apply the following result in Section \ref{spinexamples}.

\begin{theorem}\label{fishtail}
The result of performing a multiplicity 1 surgery on a fiber of a fishtail neighborhood $F$, with direction given by the vanishing cycle, is diffeomorphic to $F$ relative to the boundary.
\end{theorem}

\end{subsection}
\end{section}

\begin{section}{Spinning constructions}\label{spinexamples}

First we introduce a nice family of tori in $S^4$. Let $K$ be a knot in $S^3$. Remove from $S^3$ a 3--ball disjoint from $K$, and consider the resulting pair $(B^3, K)$. Then we get a torus $\mathcal{T}_K$ in $S^4$ by taking $K \times S^1 \subset B^3 \times S^1$ in the decomposition $B^3 \times S^1 \cup_\id S^2 \times D^2$ of the 4--sphere (see Figure \ref{spinning}). Note that we get a different torus, denoted $\mathcal{T}_K'$, if we glue $S^2 \times D^2$ to $B^3 \times S^1$ by the Gluck twist map $\rho \co S^2 \times S^1 \rightarrow S^2 \times S^1$. (Recall $\rho$ is defined by sending $(x, \theta)$ to $(rot_\theta(x), \theta)$, where $rot_\theta$ is rotation of $S^2$ about a fixed axis through angle $\theta$.) We will call $\mathcal{T}_K$ the \emph{spun torus} of $K$ and $\mathcal{T}_K'$ the \emph{twisted spun torus} of $K$ (Boyle calls $\mathcal{T}_K'$ a \emph{turned torus} \cite{Boyle}). We will choose our embedding $S^1 \times S^1 \hookrightarrow \mathcal{T}_K$ (resp. $\mathcal{T}_K'$) so that $\alpha$ is identified with $K \times \{pt\}$ and $\beta$ is identified with $\{pt\} \times S^1$ in $K \times S^1$.

For a nontrivial knot $K \subset S^3$, the exteriors of $\mathcal{T}_K$ and $\mathcal{T}_K'$ will have the same fundamental group (the knot group for $K$), but they are neither isotopic \cite{Liv} nor have diffeomorphic exteriors \cite{Boyle}. We remark that the special case of spinning torus knots in $S^3$ (and surgery on the resulting tori) was extensively studied by Iwase in \cite{Iw1} and \cite{Iw2}.

\begin{proposition}
Given a knot $K \subset S^3$, the twisted spun torus $\mathcal{T}_K'$ lies in a fishtail neighborhood in $S^4$ with the vanishing cycle given by a push off of $\beta$.
\end{proposition}

\begin{proof}
Isotope $K$ in $B^3$ so that a point $x \in K$ lies near the boundary $\partial B^3$. In particular, arrange so that $x \times \partial D^2 \subset \partial \nu K$ is tangent to $\partial B^3$ at a point $x_0 \in \partial B^3$. Now we obtain $\mathcal{T}_K'$ in $S^4$ by crossing $(B^3, K)$ with $S^1$ and gluing in $S^2 \times D^2$ by the map $\rho$. In terms of handles, we glue in $S^2 \times D^2$ by first attaching a 2-handle $h_2$ along $\{pt\}  \times S^1$ on the boundary $S^2 \times S^1$ with framing given by adding a left-handed twist to the product framing (if we glue by the identity map instead of $\rho$ we would take the product framing), and then capping with a 4-handle. We can choose the attaching circle of $h_2$ to be $x_0 \times S^1$, and then it follows directly from the definition that $(\nu K \times S^1) \cup h_2 = \nu \mathcal{T}_K' \cup h_2$ is a fishtail neighborhood. We have attached $h_2$ along a push off of $\beta$, and this attaching circle is the vanishing cycle.
\end{proof}

In contrast to the case of knots admitting non-trivial $S^3$ surgeries (only the unknot admits such a surgery \cite{GL}), we can construct infinitely many tori in $S^4$ that admit non-trivial $S^4$ surgeries (by a non-trivial surgery we mean one such that the gluing map does not extend to a diffeomorphism of $T^2\times D^2$). In particular, as a corollary of the preceding proposition we see that each twisted spun torus $\mathcal{T}_K'$ admits infinitely many non-trivial surgeries to $S^4$.

\begin{corollary}\label{S4surgeries}
$S^4_{\mathcal{T}_K'}(1,0, b)$ is diffeomorphic to $S^4$.
\end{corollary}

\begin{proof}
This follows directly from Theorem \ref{fishtail}.
\end{proof}

The author first observed the existence of these surgeries follows from a more general result appearing in unpublished work by Gompf \cite{Gom2}. Since $S^4$ does not admit a unique surgery description, it is natural to ask how widespread is this phenomenon.

\begin{question}
If $X$ is a 4--manifold obtained by surgery on a torus $\mathcal{T}$ in $S^4$, can $X$ be obtained by surgery on an infinite family of distinct tori $\{\mathcal{T}_i\}$ in $S^4$?
\end{question}

The above examples also suggest another question. In contrast with the classical dimension \cite{GL}, it is known that there exist inequivalent 2--knots with the same complement \cite{Go}. However, by \cite{gluck}  there can be at most two $2$--knots with the same complement. The case for tori in $S^4$ is unknown.

\begin{question}
Do there exist  (perhaps infinitely many) distinct tori in $S^4$ with the same complement?
\end{question}

Torus surgery appears to be the right perspective to answer this question positively. The goal would be to find non-trivial tori that admit surgeries to $S^4$ such that the surgery gluing map does not extend over the exterior or the neighborhood of the torus (this rules out the examples from Corollary \ref{S4surgeries}).

Next we define the spin of a manifold.

\begin{definition}
Let $M$ be a closed $n$--manifold, and let $M^\circ$ denote $M$ with an open ball removed. Then the \emph{spin} of $M$ is the closed $(n+1)$--manifold defined by $spin(M) = \partial (M^\circ \times D^2)$. This is equivalent to taking $M^\circ \times S^1 \cup S^2 \times D^2$.
\end{definition}

The following is an easy observation.

\begin{proposition}\label{spinning manifolds}
Let $M$ be a closed, orientable 3--manifold. Then $spin(M)$ can be obtained by surgery on a link of tori in $S^4$.
\end{proposition}

\begin{proof}
The Lickorish-Wallace theorem states that $M$ can be obtained by Dehn surgery on a link $L$ in $S^3$. Remove a 3--ball away from $L$ in $S^3$, so that we now think of $L$ as sitting in $B^3$. We obtain a link of tori $T_L$ in $S^4$ by taking $L \times S^1 \subset B^3 \times S^1$ inside $spin(S^3) = B^3 \times S^1 \cup S^2 \times D^2 = S^4$. We can now perform surgery on $T_L$ where we take the gluing maps to be $S^1$ times the Dehn surgery gluing maps on $L$ that give $M$. Hence we transform $B^3 \times S^1 \cup S^2 \times D^2$ to $M^\circ \times S^1 \cup S^2 \times D^2 = spin(M)$.
\end{proof}

\begin{figure}
\includegraphics[scale=0.25]{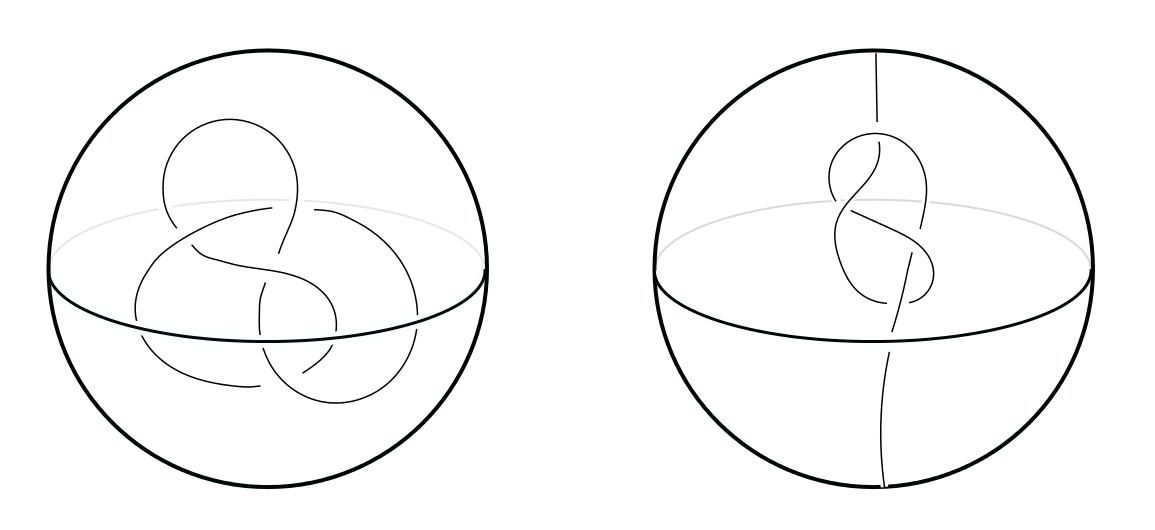}
\centering
 \caption{Spinning the knot $K \subset B^3$ on the left will result in a torus $\mathcal{T}_K$ in $S^4$. Spinning the knotted arc $\hat{K} \subset B^3$ on the right will result in the 2--knot $spin(K)$ in $S^4$.} 
\label{spinning}
\end{figure}

Here we recall the process of spinning knots, a construction due to Artin (see Figure \ref{spinning}).

\begin{definition}
Let $K$ be a knot in $S^3$. If we remove an open ball around a point of $K$, we get a knotted arc $\hat{K} \subset B^3$. The \emph{spin} of $K$ (denoted $spin(K)$) is the 2--knot obtained by taking the annulus $\hat{K} \times S^1 \subset B^3 \times S^1$ and capping off with two disks in $S^2 \times D^2$ inside the decomposition $S^4 =B^3 \times S^1 \cup_\id S^2 \times D^2$.
\end{definition}

This is a well-studied operation that generalizes to higher-dimensional knots. The following is a corollary to Proposition \ref{spinning manifolds}.

\begin{corollary}
Let $K$ be a knot in $S^3$. The $d$-fold cyclic cover of $S^4$ branched over $spin(K)$ can be obtained by surgery on a link of tori in $S^4$.
\end{corollary}

\begin{proof}
Let $M_d(K)$ denote the $d$-fold cyclic cover of $S^3$ branched over $K$. If $\widetilde{K}$ is the lift of $K$ in  $M_d(K)$, let $B$ denote a fibered neighborhood of a point $x \in \widetilde{K}$. Now if $M_d(K)^\circ = \overline{M_d(K) \setminus B}$, then $spin(M_d(K)) = M_d(K)^\circ \times S^1 \cup S^2 \times D^2$. Observe that the $d$-fold branched covering $S^2 \rightarrow S^2$ (with branch points the two poles) times the identity on the $D^2$ factor fits together with the induced branched covering on $M_d(K)^\circ \times S^1$ to form a branched covering on $spin(M_d(K))$ with branch locus $spin(K)$. Hence $spin(M_d(K))$ is the $d$-fold cyclic cover of $S^4$ branched over $spin(K)$, and the result then follows from Proposition \ref{spinning manifolds}.
\end{proof}

Next we prove a generalization of a theorem appearing in \cite{LM}, where the authors considered the case of spinning \emph{fibered} knots and gave a proof relying on interpreting monodromy changes as surgeries.

\begin{theorem}
Let $K_1$ and $K_2$ be two knots in $S^3$. Then $(S^4, spin(K_2))$ can be obtained from $(S^4, spin(K_1))$ by surgery on a link of tori in the complement of $spin(K_1)$.
\end{theorem}

\begin{proof}
Consider knots $K_1$ and $K_2$ in $S^3$. We can obtain $K_2$ from $K_1$ by surgery on a link $L$ in the exterior of $K_1$, where each component is unknotted and the framing is $\pm1$, since such surgeries allow one to change overcrossings to undercrossings and vice versa. If we remove a small ball around a point of $K_1$ we see $L$ and $\hat{K_1}$ inside $B^3$. Then upon spinning we get $spin(K_1)$ and a link $\mathcal{T}_L$ of tori in $S^4$. By construction, performing multiplicity $\pm1$ surgeries on $\mathcal{T}_L$ (using the Dehn surgery maps times the identity map in the $S^1$ direction) in the exterior of $spin(K_1)$ will return $spin(K_2)$.
\end{proof}

\end{section}

\begin{section}{Round cobordisms}\label{round}

A useful way to study torus surgeries is by round handles (for a full development of this perspective see \cite{BS}). Recall that an $n$-dimensional round $k$-handle is a copy of $S^1 \times D^k \times D^{n-1-k}$ attached along $S^1 \times \partial D^k \times D^{n-1-k}$. It is a basic fact that an $n$-dimensional round $k$-handle can be decomposed into an $n$ dimensional $k$-handle and an $n$ dimensional $(k+1)$-handle. Consider in particular the case of a 5-dimensional round 2-handle $S^1 \times D^2 \times D^2$ attached along $S^1 \times \partial D^2 \times D^2$ to $X \times I$ for some 4--manifold $X$. This defines a \emph{round cobordism} between $X$ and $X'$, where $X'$ can be obtained from $X$ by removing the attaching region $S^1 \times \partial D^2 \times D^2$ and gluing $S^1 \times D^2 \times \partial D^2$ by the identification of their boundary. Observe that this is simply an integral torus surgery on the torus $\mathcal{T} = S^1 \times \partial D^2 \times \{0\}$.
Furthermore, the converse is also true: any integral torus surgery corresponds to a cobordism given by a 5-dimensional round 2-handle (see \cite{BS} Lemma 2). The torus, multiplicity, and direction of the surgery determine how the attaching region $S^1 \times \partial D^2 \times D^{2}$ of the round 2-handle is embedded.

For our purposes the most important tool will be the Fundamental Lemma of Round Handles, due to Asimov \cite{Asimov}. The Lemma states that if we attach a $k$-handle $h_k$ and a $k+1$-handle $h_{k+1}$ to a manifold \emph{independently}, then we can combine $h_k$ and $h_{k+1}$ to form a single round $k$-handle. This means that if we form a 5-dimensional cobordism $W$ from $X$ to $X'$ by independently adding a 5-dimensional 2-handle and a 5-dimensional 3-handle to $X \times I$ (so that the attaching sphere of the 3-handle is disjoint from the belt sphere of the 2-handle), then $W$ can be decomposed as $X \times I$ plus a 5-dimensional round 2-handle. By our remarks above, we see that $X'$ can be obtained from $X$ by an integral torus surgery.

We can use this method to construct 4--manifolds that can be obtained by surgery on a link of tori in $S^4$. The following theorem is proved using a technique similar to that found in \cite{Ker}, where Kervaire gives a characterization of the fundamental groups of homology spheres of dimension greater than 4.

\begin{theorem}\label{group}
Any finitely presented group with non-negative deficiency appears as the fundamental group of a 4--manifold obtained by integral surgery on a link of tori in $S^4$.
\end{theorem}

\begin{proof}
The important observation is that if we form a 5-dimensional cobordism by first attaching a 3-handle $h_3$ and then a 2-handle $h_2$, the 2- and 3-handles will be attached independently (by transversality we can isotope the attaching sphere of the 2-handle off the belt sphere of the 3-handle and then off the 3-handle completely). Then by the Fundamental Lemma of Round Handles $h_3$ and $h_2$ can be isotoped to form a single 5-dimensional round 2-handle.

Now let $G = \langle g_1, g_2, \ldots, g_m | r_1, r_2, \ldots, r_n \rangle$ be a finitely presented group with deficiency $m-n \geq 0$. We construct a 4--manifold with fundamental group $G$ as follows. First attach $m$ 5-dimensional 3-handles to $S^4 \times I$ along attaching 2--spheres that form the $m$-component unlink in $S^4 \times \{1\}$. This gives a cobordism from $S^4$ to $\#_m S^1 \times S^3$. Note that $\pi_1 (\#_m S^1 \times S^3) \cong F_m$, the free group on $m$ letters. We can represent each relation $r_i$ of $G$ by an embedded curve $\rho_i$ in $\#_m S^1 \times S^3$, and we can assume these curves are disjoint. Then we attach $n$ 5-dimensional 2-handles along the $\rho_i$. This has the affect of surgering out a neighborhood of each $\rho_i$, which is a copy of $S^1 \times D^3$, and gluing in a copy of $D^2 \times S^2$. We see that in the resulting 4--manifold each $\rho_i$ will be nullhomotopic, and so the new fundamental group will be exactly $G$. Lastly we attach $m-n$ more 2-handles along disjoint nullhomotopic curves in the boundary. This will leave the fundamental group unchanged, but now we have a cobordism $W$ between $S^4$ and a 4--manifold $X$ formed by first attaching $m$ 3-handles, and then attaching $m$ 2-handles. By our comments above, we see that these handles must be attached independently, and so $W$ is formed by attaching $m$ round 2-handles to $S^4 \times I$ (whose attaching regions are disjoint). Therefore $X$ is obtained by integral surgery on a link of tori in $S^4$, and $\pi_1 (X) \cong G$.
\end{proof}

For groups of negative deficiency we have the following result.

\begin{proposition}
One can produce 4--manifolds with fundamental groups of arbitrarily large negative deficiency by surgery on tori in $S^4$.
\end{proposition}

\begin{proof}
Start with the family $\mathcal{K}_m$ of 2--knots constructed in \cite{Levine}. Levine showed that the knot group of $\mathcal{K}_m$ has deficiency $1-m$. Then we can produce a family of tori $\mathcal{T}_m$ by adding a trivial tube to each $\mathcal{K}_m$ in a small 4--ball neighborhood of a point in $\mathcal{K}_m$ (see Figure \ref{tube}). This doesn't change the fundamental group of the exterior, and the $T^3$ boundary of $E_{\mathcal{T}_m}$ will have two $S^1$ factors that are nullhomotopic in the exterior (the third $S^1$ factor is the meridinal direction). We can then perform a multiplicity 0 torus surgery on $\mathcal{T}_m$ with surgery direction either of the nullhomotopic $S^1$ factors of the boundary. The result of this surgery will have the same fundamental group as the exterior since the attaching curve of the 2-handle is already nullhomotopic in the exterior, and so we get groups with arbitrarily large negative deficiency.
\end{proof}

\begin{figure}
\includegraphics[scale=0.3]{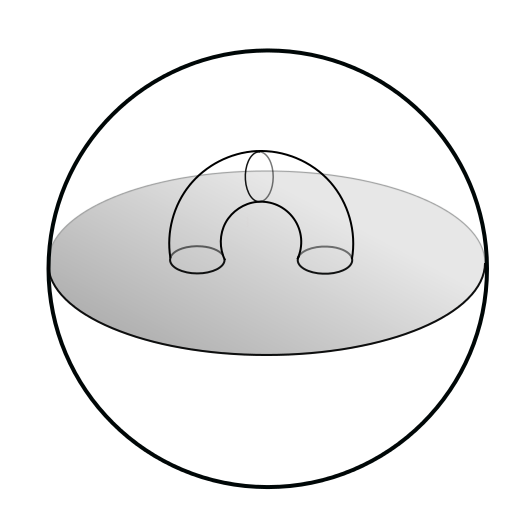}
\centering
 \caption{Adding a trivial tube to a 2--knot. Here we see a a slice $B^3 \times \{1/2\}$ of a 4--ball neighborhood $B^3 \times I$ of a point on a 2--knot. The disk bounded by the equator is a small patch of the 2--knot, and we add a small tube to increase the genus of the surface. Notice that each of the $S^1$ factors bounds a disk in the exterior.}
\label{tube}
\end{figure}


However, it is not possible to achieve all finitely presented groups. Hausmann and Weinberger \cite{HW} constructed a finitely presented group that cannot be realized as the fundamental group of a 4--manifold with Euler characteristic 2. Therefore it seems difficult to give a complete characterization of the possible fundamental groups of 4--manifolds obtained by torus surgery in $S^4$.

We can use round handles to give another description of 4--manifolds obtained by torus surgery in $S^4$. Here \emph{spherical surgery} means replacing an embedded copy of $S^2 \times D^2$ with $S^1 \times D^3$. 

\begin{proposition}
Let $X$ be a 4--manifold obtained by integral surgery on a single torus in $S^4$. Then $X$ can also be obtained by a spherical surgery on an embedded $S^2$ in $S^2 \times S^2$ or $S^2 \mathbin{\widetilde{\times}} S^2$.
\end{proposition}

Recall that $S^2 \mathbin{\widetilde{\times}} S^2$ is the twisted $S^2$ bundle over $S^2$, and is diffeomorphic to $\mathbb{CP}^2 \# \overline{\mathbb{CP}^2}$.

\begin{proof}
We saw above that $X$ can be obtained by attaching a 5-dimensional round 2-handle $R$ to $S^4 \times I$. Furthermore, $R$ can be decomposed into a 2-handle $h_2$ and 3-handle $h_3$. A 5-dimensional 2-handle is a copy of $D^2 \times D^3$ attached along $\partial D^2 \times D^3$. Up to isotopy there is a unique circle in $S^4$, and there are two surgeries on this circle corresponding to a $\mathbb{Z}_2$ choice of framing. The resulting 4--manifolds are $S^2 \times S^2$ and $S^2 \mathbin{\widetilde{\times}} S^2$. Hence $h_2$ gives a cobordism from $S^4$ to $M$, where $M$ is one of the $S^2$ bundles over $S^2$. We complete our cobordism to $X$ by attaching the 3-handle $h_3$ to $M$. A 5-dimensional 3-handle is a copy of $D^3 \times D^2$ attached along $\partial D^3 \times D^2$. Observing how the boundary changes, we see that $X$ is obtained by spherical surgery on $\partial D^3 \times \{0\}$ in $M$, and so the result follows.

\end{proof}

\begin{remark}
Using similar techniques it is not hard to show that if $X$ is the result of surgery on a \emph{link} of tori, then $X$ can be obtained by a set of spherical surgeries in $\#_k S^2 \times S^2$ or $\#_k S^2 \mathbin{\widetilde{\times}} S^2$ for some $k$. For non-integral surgeries we use the fact pointed out in \cite{BS} that the result of a non-integral surgery can be obtained as a set of integral surgeries.
\end{remark}

\end{section}


\begin{section}{The unknotted torus and embedding 3--manifolds}\label{unknot}

The unknotted torus is the unique torus in $S^4$ that bounds a solid torus $S^1 \times D^2$.
In \cite{Mont1}, Montesinos analyzed which gluing diffeomorphisms extend over the exterior of the unknotted torus (this exterior is the so-called \emph{standard twin}). He used this to prove the following theorem; here instead we give a short proof using Cerf's theorem, which states that any self-diffeomorphism of $S^3$ can be extended to a self-diffeomorphism of $B^4$.

\begin{theorem}\label{mult1}\cite{Mont1}
Let $\mathcal{T}$ be the unknotted torus in $S^4$. Then the result of any multiplicity 1 surgery on $\mathcal{T}$ is diffeomorphic to $S^4$.
\end{theorem}

\begin{proof}
Fix a particular multiplicity 1 surgery on $\mathcal{T}$, and let $q$ and $\gamma$ be the corresponding auxiliary multiplicity and direction of the surgery. Since $\mathcal{T}$ is unknotted, we can isotope it so it lies embedded in the standard $S^3$ equator of $S^4$.
Then a neighborhood of $\mathcal{T}$ in $S^4$ is $\mathcal{T} \times I_0 \times I_1$, where $\mathcal{T} \times I_0$ is a neighborhood of $\mathcal{T}$ in $S^3$ and $I_1 = [0,1]$ is the interval induced from a collar neighborhood of $S^3$ in $S^4$. We perform the surgery as follows. We can remove $\mathcal{T} \times I_0 \times I_1$ and re-glue by any diffeomorphism of the boundary with multiplicity 1, auxiliary multiplicity $q$, and direction $\gamma$. Therefore we can choose the gluing map to be the identity map on $\mathcal{T} \times \partial I_0 \times I_1 \cup \mathcal{T} \times I_0 \times \{1\}$ and on $\mathcal{T} \times I_0 \times \{0\}$ to be the map that twists $\mathcal{T}$ in the $\gamma$ direction $q$ times. Finally, we observe that this surgery is equivalent to cutting $S^4$ along $S^3 \times \{0\}$ and re-gluing by the diffeomorphism of $S^3$ given by twisting $\mathcal{T}$ in the $\gamma$ direction $q$ times. By Cerf's theorem this diffeomorphism extends over $B^4$ and so we get back $S^4$.
\end{proof}

Now we will use Montesinos' work to show that the result of multiplicity 0 surgery on the unknotted torus is also quite restrictive, although as we saw in Proposition \ref{spinprop} we should obtain both spin and non-spin manifolds. This proposition was proved by Pao \cite{Pao} (at least in the topological category) in the context of torus actions on 4--manifolds. Iwase also gives a proof in \cite{Iw3}; in fact, Iwase gives a similar classification for surgery on the unknotted torus for any multiplicity.

\begin{theorem}\label{mult0}
The result of multiplicity 0 surgery on the unknotted torus $\mathcal{T}$ is either $S^1 \times S^3 \# S^2 \times S^2$ or $S^1 \times S^3 \# S^2 \mathbin{\widetilde{\times}} S^2$. Indeed, 
$S^4_\mathcal{T}(0,a,b)$ is diffeomorphic to $S^1 \times S^3 \# S^2 \times S^2$ if $ab$ is even, and $S^1 \times S^3 \# S^2 \mathbin{\widetilde{\times}} S^2$ if $ab$ is odd.
\end{theorem}

Here we choose our embedding of $\mathcal{T}$ by realizing the unknotted torus as the spin of the unknot $U \subset S^3$, so that the first $S^1$ factor $\alpha$ is identified with a longitude of $U$ and the other $S^1$ factor $\beta$ is identified with the $S^1$ direction of the spin.


\begin{proof}
Montesinos \cite{Mont1} showed that gluing maps of the form
\[\psi = \left( \begin{array}{ccc}
c & d & * \\
e & f & * \\
0 & 0 & 1 \end{array} \right)\]
(where $c+d+e+f$ is an even number) extend over the exterior of the unknotted torus and hence don't affect the resulting diffeomorphism type. We will show that any choice of $a$ and $b$ (necessarily relatively prime) can be obtained by starting with a gluing map that has direction $\gamma =  [\alpha]$ or $\gamma = [\alpha] + [\beta]$ and then post-composing with a gluing map of the form $\psi$ above. Then by Montesinos' result we see that $\psi$ extends over the exterior and so the two possible resulting diffeomorphism types are $S^4_\mathcal{T}(0,1,0)$ and $S^4_\mathcal{T}(0,1,1)$. In the following lemma we will show that these manifolds are diffeomorphic to $S^1 \times S^3 \# S^2 \times S^2$ and $S^1 \times S^3 \# S^2 \mathbin{\widetilde{\times}} S^2$, respectively.

Suppose the direction of the surgery is $\gamma = a[\alpha] + b[\beta] \in H_2(\mathcal{T})$. Then we can choose the gluing map to be
\[\phi = \left( \begin{array}{ccc}
0 & m & a \\
0 & n & b \\
1 & 0 & 0 \end{array} \right)\]
for some integers $m$ and $n$ satisfying $mb-na = 1$. Post-composing with a map $\psi$ as above has the effect of changing the direction 
$\begin{pmatrix}
  a \\
  b
 \end{pmatrix}$
by multiplying by the  even matrix 
$\begin{pmatrix}
  c & d \\
  e & f
 \end{pmatrix}$
(we will say a matrix is \emph{even} if the sum of its entries is even). In light of this it is sufficient to show that for any relatively prime pair
$\begin{pmatrix}
  a \\
  b
 \end{pmatrix}$
there is an even matrix $A$ such that if $ab$ is an odd number then we have 
$\begin{pmatrix}
  a \\
  b
 \end{pmatrix} = A \begin{pmatrix}
  1 \\
  1
 \end{pmatrix}$
and if $ab$ is an even number then
$\begin{pmatrix}
  a \\
  b
 \end{pmatrix} = A \begin{pmatrix}
  1 \\
  0
 \end{pmatrix}$.
First assume that $ab$ an odd number. Then $a + b$ is an even number. There exists some matrix in $GL(2, \mathbb{Z})$ such that
$\begin{pmatrix}
  a \\
  b
 \end{pmatrix} =
\begin{pmatrix}
  c & d \\
  e & f
 \end{pmatrix}
\begin{pmatrix}
  1 \\
  1
 \end{pmatrix}=
\begin{pmatrix}
  c+d \\
  e+f
 \end{pmatrix}$.
Since $a+b = (c+d) + (e+f)$ is even, we see that this matrix is even.

Now assume that $ab$ is an even number. Then $a+b$ is an odd number. Now we can write
$\begin{pmatrix}
  a \\
  b
 \end{pmatrix} =
\begin{pmatrix}
  a & -d \\
  b &c
 \end{pmatrix}
\begin{pmatrix}
  1 \\
  0
 \end{pmatrix}$
for some integers $c$ and $d$ that are solutions to the equation $ax + by = 1$. Given one such pair of solutions $(c, d)$, it is known that all other solutions are of the form $(c + kb, d-ka)$ for some integer $k$. In particular, this implies that it is always possible to choose a pair of solutions with opposite parity (since $a$ and $b$ necessarily have opposite parity), and so we can choose the above matrix to be even.

To complete the proof we need to show that $S^4_\mathcal{T}(0,1,0)$ is diffeomorphic to $S^1 \times S^3 \# S^2 \times S^2$ and $S^4_\mathcal{T}(0,1,1)$ is diffeomorphic to $S^1 \times S^3 \# S^2 \widetilde{\times} S^2$. We will do in the following lemma. 
\end{proof}

\begin{figure}
\includegraphics[scale=0.3]{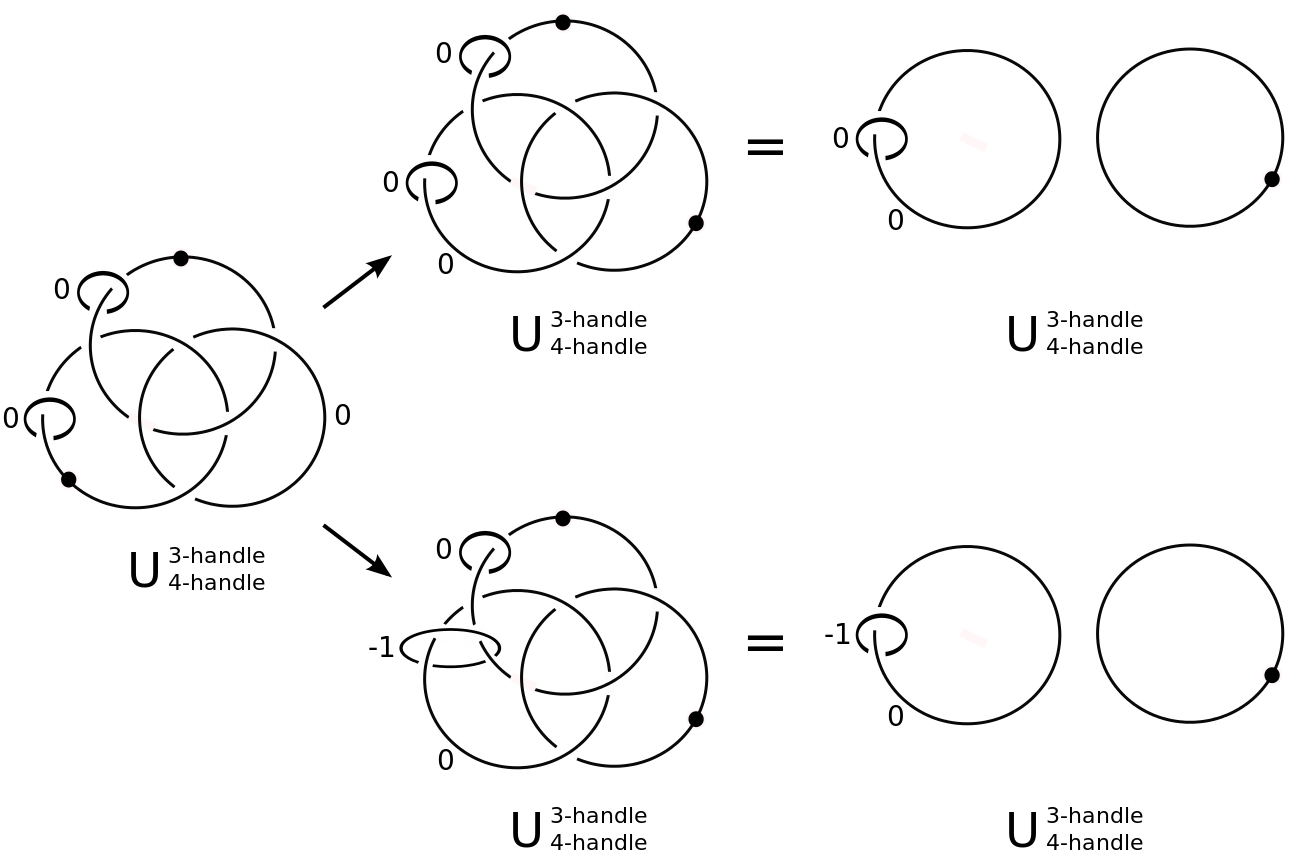}
\centering
 \caption{Multiplicity 0 surgery on the unknotted torus.}
\label{surgery}
\end{figure}

\begin{lemma}
For the unknotted torus $\mathcal{T} \subset S^4$, we have $S^4_\mathcal{T}(0,1,0)$ is diffeomorphic to $S^1 \times S^3 \# S^2 \times S^2$ and $S^4_\mathcal{T}(0,1,1)$ is diffeomorphic to $S^1 \times S^3 \# S^2 \widetilde{\times} S^2$.
\end{lemma}

\begin{proof}
We prove this using a Kirby calculus description of torus surgery (see \cite{GS} Chapter 8 for a thorough explanation of this perspective). On the left hand of Figure \ref{surgery} we see the unknotted torus embedded in $S^4$. The Borromean rings consisting of two 1-handles in dotted circle notation and one 0 framed 2-handle give a handle decomposition of $\nu \mathcal{T}$ (here we clearly see the boundary is $T^3$), and the rest of the handles give the embedding into $S^4$. The top row of Figure \ref{surgery} corresponds to doing multiplicity 0 surgery with direction $[\alpha]$. We cut out $\nu \mathcal{T}$, and re-glue by the diffeomorphism of $T^3$ that cyclically permutes the three $S^1$ factors, giving the middle top picture. A handle cancellation results in $S^1 \times S^3 \# S^2 \times S^2$, proving the first part of the lemma.

The bottom row of Figure \ref{surgery} corresponds to multiplicity 0 surgery with direction $[\alpha] + [\beta]$. Here we cut out $\nu \mathcal{T}$ and re-glue by the diffeomorphism of $T^3$ that first cyclically permutes the three $S^1$ factors and then applies $\phi$, where $\phi$ is the map that cuts along $T^2 \times \{pt\} \subset T^2 \times \partial D^2$ and re-glues by a map that puts a full rotation in the direction of the first $S^1$ factor and is the identity on the second factor (this is analogous to a Dehn twist). Observe that $\phi$ sends any curve intersecting $T^2 \times \{pt\}$ to a curve that also wraps around the first $S^1$ factor, and adds a twist to the framing. The result is the bottom middle picture, and handle slides and a cancellation show that this is diffeomorphic to $S^1 \times S^3 \# S^2 \widetilde{\times} S^2$.

%
%
\end{proof}

\begin{subsection}{Embeddings}

We can now use Theorem \ref{mult0} to prove a theorem about embedding 3--manifolds obtained by surgery on a knot in $S^3$.

\begin{theorem}
Let $K$ be a knot in $S^3$ and let $S^3_{p/q}(K)$ denote $p/q$ Dehn surgery on $K$. Then $S^3_{p/q}(K)$ smoothly embeds in $S^1 \times S^3 \# S^2 \times S^2$ if $pq$ is even, and embeds in $S^1 \times S^3 \# S^2 \widetilde{\times} S^2$ if $pq$ is odd.
\end{theorem}

\begin{proof}
Consider $K$ as sitting in the standard $S^3$ equator of $S^4$. Then $\partial \nu K \subset S^3$ gives an unknotted torus $\mathcal{T}$ when we include $S^3 \hookrightarrow S^4$. We do multiplicity 0 surgery on $\mathcal{T}$ as follows. We will choose our embedding of $\mathcal{T}$ so that the first $S^1$ factor is the meridian of $K$ and the second $S^1$ factor is the longitude of $K$. The collar neighborhood of $\partial \nu K$ in $ S^3$ and the collar neighborhood of $S^3$ in $S^4$ provide a framing for $\mathcal{T}$. We will choose our surgery direction $\gamma$ to be the homology class of $p$ times the meridian and $q$ times the longitude. We claim that $S^3_{p/q}(K)$ embeds in $S^4_\mathcal{T}(0,p,q)$, which by Theorem \ref{mult0} is diffeomorphic to $S^1 \times S^3 \# S^2 \times S^2$ if $pq$ is even, and $S^1 \times S^3 \# S^2 \widetilde{\times} S^2$ if $pq$ is odd.

Our goal is to see the Dehn surgery concurrently with the torus surgery. To obtain $S^4_\mathcal{T}(0,p,q)$ we remove our chosen neighborhood of $\mathcal{T}$ and glue back $S^1 \times S^1 \times D^2$ by a gluing map of the form
\[\phi = \left( \begin{array}{ccc}
0 & m & p \\
0 & n & q \\
1 & 0 & 0 \end{array} \right)\]
(for suitable $m$ and $n$).
Hence we are gluing in a solid torus $\{pt\} \times S^1 \times D^2$ to each $S^1 \times S^1 \times \{pt\} \subset S^1 \times S^1 \times \partial D^2$, where the boundary of the meridinal disk gets mapped to $p$ times the first factor and $q$ times the second factor. Since $\partial E_K \subset \partial E_{T}$ is exactly one of these tori, the result follows.

\end{proof}

Observe that the meridian of $\mathcal{T}$ intersects $E_K \subset S^4$ in a single point. Hence this curve intersects $S^3_{p/q}(K)$ in a single point after the torus surgery, and this curve generates the fundamental group of the resulting 4--manifold. It follows that surgery on this curve (replacing $S^1 \times D^3$ with $D^2 \times S^2$) will kill the $S^1 \times S^3$ connected summand of the resulting 4--manifold. This surgery will puncture $S^3_{p/q}(K)$, and so we get the following corollary.

\begin{corollary}
If $S^3_{p/q}(K)^\circ$ denotes the 3--manifold obtained by puncturing $S^3_{p/q}(K)$, then $S^3_{p/q}(K)^\circ$ smoothly embeds in $S^2 \times S^2$ if $pq$ is even, and embeds in $ S^2 \widetilde{\times} S^2$ if $pq$ is odd.
\end{corollary}

See \cite{EL} for a similar statement when  $K$ is the unknot (and so the 3--manifolds are lens spaces). Note that any 3--manifold obtained by \emph{integral} surgery on a knot always embeds in $S^2 \times S^2$ or $ S^2 \widetilde{\times} S^2$ (just double the 4--dimensional 2-handlebody). An interesting thing about the above construction is that it does not distinguish between integral and non-integral Dehn surgery.

Next we look at embeddings of 3--manifolds into $S^4$. While many such constructions depend on handlebody techniques and branched covers of doubly slice knots (for example, see \cite{Don}, \cite{Gil-Liv}, \cite{Meier}), here we take an alternative approach using surgery.

\begin{theorem}
Let $L$ be a ribbon link in $S^3$. If $M_L$ is the 3--manifold obtained by surgery on $L$ with all the surgery coefficients belonging to the set $\{1/n\}_{n \in \mathbb{Z}}$, then $M_L$ smoothly embeds in $S^4$. If $L$ is only \emph{slice}, then we get an embedding into a \emph{homotopy} 4--sphere. However, if we restrict the surgery coefficients to the set $\{1/(2n)\}_{n \in \mathbb{Z}}$, then again we get an embedding into the standard $S^4$.
\end{theorem}

Budney and Burton \cite{BB} observed that if $L$ is slice and the coefficients are $\pm 1$, then $M_L$ embeds in a homotopy 4--sphere by blowing down the resulting 2--spheres in the 2-handlebody formed by attaching 2-handles to $L$ with the corresponding framings. We obtain this generalization by proceeding in a different direction; we consider cross sections of Gluck twists on the 2--knots obtained by doubling the ribbon or slice disks.

\begin{proof}

First we consider the case where $K$ is a slice knot in $S^3$. Let $\mathcal{D}_K$ be the slice disk in $B^4$, and let $\mathcal{S}_K$ be the 2--knot in $S^4$ obtained by doubling the pair $(B^4, \mathcal{D}_K)$. We will do surgery on $\mathcal{S}_K$ and see Dehn surgery on $K$ as a cross section. Identify a neighborhood of $\mathcal{S}_K$ with $S^2 \times D^2$ such that $equator \times D^2$ is identified with a tubular neighborhood of $K$ in $S^3 \subset S^4$ (and the induced framing is the zero framing). We will cut out $S^2 \times D^2$ and re-glue by the map $\rho \co S^2 \times \partial D^2 \rightarrow  S^2 \times \partial D^2$ defined by sending $(x, \theta)$ to $(rot_\theta(x), \theta)$, where $rot_\theta$ is the map that rotates $S^2$ through an angle $\theta$ about a fixed axis (we choose this to send the equator to itself). Now the result of this surgery is by definition the Gluck twist on $\mathcal{S}_K$ in $S^4$, and hence returns a homotopy 4--sphere. In fact this is true for all odd powers of $\rho$. However, if we instead re-glue by an even power $\rho^{2n}$ of $\rho$, then since $\rho^2$ is isotopic to the identity map \cite{gluck} the result of the surgery will be the \emph{standard} $S^4$. Furthermore, if $D_K$ is a \emph{ribbon} disk, then $\mathcal{S}_K$ is a ribbon 2--knot and re-gluing by any power $\rho^{n}$ will return the standard $S^4$ (see, for example, \cite{GS}).

Now we examine what happens to a neighborhood of $K$. For a point $x$ in $K$ (thinking of $K$ as the equator of $\mathcal{S}_K$), consider the effect of $\rho^n$ on the boundary of the meridinal disk $x \times D^2$. As $\theta$ varies the curve $(x, \theta)$ on the boundary will map to a curve that wraps once around the meridinal direction and $n$ times around the longitudinal direction of $K \times \partial D^2$. This is exactly Dehn surgery on $K$ with surgery coefficient $1/n$, and so we get an embedding into the 4--manifold obtained by the corresponding surgery on $\mathcal{S}_K$. We can extend this to surgery on a slice link by performing surgery on each of the 2--knots obtained by doubling the multiple slice disks. Finally, we finish by applying the comments in the previous paragraph about the result of the various 2--knot surgeries.

\end{proof}

As mentioned in the introduction, we note that the above theorem provides many embeddings of integral homology 3--spheres into $S^4$. Indeed, it is not hard to show that any 3--manifold that is obtained by surgery on a slice link with all the surgery coefficients belonging to the set $\{1/n\}_{n\in \mathbb{Z}}$ is an integral homology 3--sphere.

\end{subsection}
\end{section}

%
%
%
%
%
%
%
%
%
%
%

\bibliographystyle{amsalpha}
\bibliography{torussurgery.bib}

\end{document}